\documentclass[12pt]{article}
\usepackage{amsfonts,amssymb,amsbsy,amsmath,amsthm}
\usepackage[sans]{dsfont}
\usepackage{graphicx}
\usepackage{color}
\usepackage{pst-all}
\usepackage{pstricks}
\usepackage{pst-plot}
\usepackage{auto-pst-pdf}

\numberwithin{equation}{section}

\usepackage[english]{babel}
\newtheorem{theorem}{Theorem}[section]

\newtheorem{lemma}{Lemma}[section]
\newtheorem{follow}{Corollary}[section]

\newtheorem{pr}{Proposition}[section]
\theoremstyle{definition}

\newcommand{\bel}{\begin{equation} \label}
\newcommand{\ee}{\end{equation}}

\newcommand{\one}{\mathds{1}}
\newcommand{\Op}{{\rm Op}^{\rm w}}
\newcommand{\Oph}{{\rm Op}^{\rm w}_\hbar}
\newcommand{\lo}{\ell_+}

\newcommand{\Z}{{\mathbb Z}}
\newcommand{\R}{{\mathbb R}}
\newcommand{\N}{{\mathbb N}}
\newcommand{\C}{{\mathbb C}}
\newcommand{\re}{{\mathbb R}}

\newcommand{\rd}{{\mathbb R}^{2}}
\newcommand{\tr}{{\rm Tr}}
\newcommand{\V}{{\mathbb V}}
\newcommand{\uo}{\mathaccent'27{u}}
\newcommand{\vo}{\mathaccent'27{\mathbb V}}

\begin{document}
\begin{center}{\Large \bf A Trace Formula for Long-Range Perturbations of the Landau Hamiltonian}

\medskip

{\sc Tom\'as Lungenstrass, Georgi Raikov}

\medskip
\today
\end{center}

\bigskip

\noindent
{\bf Abstract.} {\small We consider the Landau Hamiltonian  perturbed by a long--range electric
potential $V$. The spectrum of the perturbed operator consists of eigenvalue clusters which accumulate to the Landau levels. First, we estimate the rate of the shrinking of these clusters to the Landau levels as the number of the cluster  tends to infinity. Further, we assume that there exists an appropriate $\V$, homogeneous of order $-\rho$ with $\rho \in (0,1)$,  such that $V(x) = \V(x) + O(|x|^{-\rho - \varepsilon})$, $\varepsilon > 0$, as $|x| \to \infty$, and investigate the asymptotic distribution of the eigenvalues within the $q$th cluster as $q \to \infty$.  We obtain an
explicit description of the asymptotic density of the eigenvalues
in terms of the mean--value transform of  $\V$.}\\



\section{Introduction} \label{s1}
Our unperturbed operator is the Landau Hamiltonian
$$
H_0 : = (-i\nabla - A)^2,
$$
self--adjoint in $L^2(\rd)$. Here, $A: = \left(-\frac{Bx_2}{2}, \frac{Bx_1}{2}\right)$ is the magnetic potential, and $B>0$ is
 the  generated constant magnetic field. It is well known that the spectrum $\sigma(H_0)$ of $H_0$ consists of infinitely degenerate eigenvalues
 $\lambda_q : = B(2q+1)$, $q \in {\mathbb Z}_+ : = \{0,1,2,\ldots\}$, called {\em Landau levels}.\\
 % Denote by $P_q$, $q \in {\mathbb Z}_+$, the orthogonal projection onto ${\rm Ker}\,(H_0 - \lambda_q)$.\\
The perturbation of $H_0$ is an electric potential $V : \rd \to \re$ which is supposed to be bounded and continuous. Set $H : = H_0 + V$. Evidently,
$$
\sigma(H) \subset \bigcup_{q=0}^\infty \left[\lambda_q + \inf{V}, \lambda_q + \sup{V}\right].
$$
Moreover, if $V$ decays at infinity, and, hence, is relatively compact with respect to $H_0$, then $\sigma(H)\setminus \sigma(H_0)$ consists of discrete eigenvalues which could accumulate only to the Landau levels. Recently, in \cite{prvb} it was shown that if $V$ satisfies
\bel{25}
    |V(x)| \leq c \langle x \rangle^{-\rho}, \quad x \in \rd,
    \ee
    with $\rho > 1$, then $\sigma(H)$ is contained in the union of intervals centered at the Landau levels $\lambda_q$, of size $O(\lambda_q^{-1/2})$ as $q \to \infty$. Moreover, in \cite{prvb} the asymptotic density of the eigenvalue clusters was studied. To this end, the asymptotic behaviour of the trace $\tr\,\varphi(\lambda_q^{1/2}(H - \lambda_q))$ with $\varphi \in C_0^\infty(\re\setminus\{0\})$ was investigated, and it was found that $\tr\,\varphi(\lambda_q^{1/2}(H - \lambda_q))$ is of order $\lambda_q^{1/2}$ as $q \to \infty$, and its first asymptotic term could be written explicitly using the Radon transform of $V$. \\
    \noindent
    In the present article we assume that $V$ is long--range, i.e. in contrast to \cite{prvb}, it satisfies \eqref{25} with $\rho \in (0,1)$. First, we show that the eigenvalue clusters of $H$ shrink to the Landau levels at rate $\lambda_q^{-\rho/2}$ as $q \to \infty$ (see Proposition \ref{p1}). Further, we suppose that there exists an appropriate $\V$, homogeneous of order $-\rho$, which is asymptotically equivalent to $V$, and study the asymptotic behaviour of the trace $\tr\,\varphi(\lambda_q^{\rho/2}(H - \lambda_q))$. We show that $\tr\,\varphi(\lambda_q^{\rho/2}(H - \lambda_q))$ is of order $\lambda_q$ as $q \to \infty$, and its main asymptotic term could be written explicitly using the mean--value transform of $\V$ (see Theorem \ref{th1}). \\
    \noindent
    The article is organized as follows. In the next section we formulate our main results, and briefly comment on them. Section \ref{s3} contains auxiliary facts concerning the properties of Weyl pseudodifferential operators and Berezin--Toeplitz operators which are the main tools in the proof of Theorem \ref{th1}. The proof itself could be found in Section \ref{s5}, and is divided into several steps, contained in separate subsections.

    \section{Main results} \label{s2}
    Our first result concerns the shrinking of the eigenvalue clusters of $H$ in the case of long--range potentials $V$.
\begin{pr} \label{p1}
Assume that $V$ satisfies \eqref{25}
    with $\rho \in (0,1)$. Then there exists a constant $C>0$ such that
    \bel{23}
    \sigma(H) \subset \bigcup_{q=0}^\infty \left(\lambda_q - C\lambda_q^{-\rho/2}, \lambda_q + C\lambda_q^{-\rho/2}\right).
    \ee
\end{pr}
\noindent
The proof of Proposition \ref{p1} could be found in Subsection \ref{ss41}.\\
\noindent
{\em Remarks}: (i) Simple considerations (see the remark after Proposition \ref{p2}) show that the estimate $O(\lambda_q^{-\rho/2})$ of the size of the eigenvalue clusters is sharp. This will follow also from Theorem \ref{th1}. \\
\noindent
(ii) In \cite[Proposition 1.1]{prvb} it was shown that if $V$ satisfies \eqref{25} with $\rho > 1$, then there exists a constant $C>0$ such that
$$
\sigma(H) \subset \bigcup_{q=0}^\infty \left(\lambda_q - C\lambda_q^{-1/2}, \lambda_q + C\lambda_q^{-1/2}\right).
$$
In the case of compactly supported $V$, such a result was already obtained in \cite{korpu}. \\

\noindent
In order to formulate our main result
we need the following notations. For $d \geq 1$ put
$$
C_{\bf b}^\infty(\re^d) = \left\{u \in C^{\infty}(\re^d) \, | \, \sup_{x \in \re} | D^{\alpha}u(x)| \leq c_\alpha, \; \alpha \in {\mathbb Z}_+^d \right\}.
$$
Following \cite[Chapter 3, Section 8]{tay1}, we write $u \in {\mathcal H}_{-\rho}^\sharp(\re^d)$ if $u \in C^{\infty}(\re^d\setminus\{0\})$, $\rho \in (0,\infty)$, is  a homogeneous function of order $-\rho$. Moreover, for $\rho \in [0,\infty)$ we set
$$
{\mathcal S}^{-\rho}_1(\re^d)  : = \left\{u \in C^{\infty}(\re^d) \, | \, \sup_{x \in \re^d}\langle x \rangle^{\rho +|\alpha|} |D^{\alpha}u(x)| \leq c_\alpha, \; \alpha \in {\mathbb Z}_+^d\right\}.
$$
Assume $u \in C(\rd \setminus \{0\})$, and define {\em the mean--value transform}
$$
\uo(x) : = \frac{1}{2\pi} \int_{{\mathbb S}^1} u(x-\omega) d\omega, \quad x \in \rd \setminus {\mathbb S}^1.
$$
Our mean--value transform coincides with the 2D mean--value operator $M^1$ defined in  \cite[Chapter I, Eq. (15)]{helg1} with $n=2$, and is quite closely related to the so--called planar circular Radon transform defined, for instance, in \cite{amku}.  \\
Next, we describe some elementary but yet useful properties of the mean--value transforms of functions from appropriate classes. The proofs are quite simple, so that we omit the details.
 If $u \in {\mathcal S}_1^{-\rho}(\rd)$, $\rho \in (0,\infty)$, then the mean-value transform $\uo$ extends to a function $\uo \in {\mathcal S}_1^{-\rho}(\rd)$.
If $u \in {\mathcal H}_{-\rho}^\sharp(\rd)$, $\rho \in (0, \infty)$, then $\eta \uo \in {\mathcal S}_1^{-\rho}$ provided that $\eta \in {\mathcal S}_1^0(\rd)$ and ${\rm supp}\, \eta \cap {\mathbb S}^1 = \emptyset$.
Moreover, if $\rho \in (0,1)$, then the mean-value transform of $u \in {\mathcal H}_{-\rho}^\sharp(\rd)$, $\rho \in (0, 1)$, extends to a function $\uo \in C(\rd)$. Finally, if $u \in {\mathcal H}_{-\rho}^\sharp(\rd)$, $\rho \in (0, 1)$, and $\uo(x) = 0$ for each $x \in \rd$, then $u(x) = 0$ for each $x \in \rd \setminus \{0\}$.
\begin{theorem} \label{th1} Let $\rho \in (0,1)$. Assume that $V \in {\mathcal S}_1^{-\rho}(\rd)$ and there exists $\V \in {\mathcal H}_{-\rho}^\sharp(\rd)$ such that
    \bel{22}
    |V(x) - \V(x)| \leq C |x|^{-\rho - \varepsilon}, \quad x \in \rd, \quad |x| > 1,
    \ee
    with some constant $C$  and $\varepsilon > 0$.
     Then we have
     \bel{24}
     \lim_{q \to \infty} \lambda_q^{-1} {\rm Tr}\, \varphi(\lambda_q^{\rho/2} (H - \lambda_q)) = \frac{1}{2\pi B} \int_{\rd} \varphi(B^{\rho} \vo(x)) dx
     \ee
     for each $\varphi \in C_0^{\infty}(\re\setminus\{0\})$.
     \end{theorem}
     Let us comment briefly on Theorem \ref{th1}.
     \begin{itemize}
     \item Let $[\alpha, \beta] \subset \re \setminus \{0\}$ be a bounded interval with $\alpha < \beta$. For $q \in \Z_+$ set
     $$
     \mu_q([\alpha, \beta]) : = \sum_{\lambda_q + \alpha \lambda_q^{-\rho/2} \leq \lambda \leq \lambda_q + \beta \lambda_q^{-\rho/2}} {\rm dim \, Ker}\,(H-\lambda).
     $$
     Evidently, $\mu_q([\alpha, \beta]) < \infty$ if $q \in \Z_+$ is large enough. Put
     $$
      \mu([\alpha, \beta]) : = \frac{1}{2\pi B} \left| \left\{x \in \rd \, | \, \alpha B^{-\rho} \leq \vo(x) \leq \beta B^{-\rho}\right\}\right|
      $$
      where $| \cdot |$ denotes the Lebesgue measure, and $\vo$ is the mean-value transform of $\V \in {\mathcal H}_{-\rho}^{\sharp}(\rd)$, $\rho \in (0,1)$, the homogeneous function introduced in the statement of Theorem \ref{th1}. Evidently, $0 \not \in [\alpha, \beta]$ implies $\mu([\alpha, \beta]) < \infty$. We extend $\mu$ to a $\sigma$-finite measure defined on the Borel sets ${\mathcal O} \subset \re \setminus \{0\}$, and supported on $\left[B^\rho \inf_{x \in \re^2} \vo(x), B^\rho \sup_{x \in \re^2} \vo(x)\right] \setminus \{0\}$; the compactness of the support of the limiting measure $\mu$ agrees with the fact that, in accordance, with \eqref{23}, we have
      ${\rm supp}\,\mu_q \subset [-C,C] \setminus \{0\}$ for sufficiently large $q$.
      Then
      the validity of \eqref{24} for any $\varphi \in C_0^{\infty}(\rd \setminus \{0\})$ is equivalent to the validity of
      $$
      \lim_{q \to \infty} \lambda_q^{-1} \mu_q([\alpha, \beta]) = \mu([\alpha, \beta])
      $$
      for any bounded $[\alpha, \beta] \subset \re \setminus \{0\}$ such that $\mu(\{\alpha\}) = \mu(\{\beta\}) = 0$. Note that if, for instance, the function $\V$ is radially symmetric, then $\mu(\{\alpha\}) = 0$ for any $\alpha \in \re \setminus \{0\}$.
      \item As already mentioned, in \cite{prvb} it was supposed  that $V$ satisfies \eqref{25} with $\rho > 1$. Then the Radon transform
      $$
      \tilde{V}(s,\omega) : = \frac{1}{2\pi} \int_\re V(s\omega + t \omega^\perp) dt, \quad s \in \re, \quad \omega = (\omega_1, \omega_2) \in {\mathbb S}^1, \quad \omega^{\perp} : = (-\omega_2, \omega_1),
       $$
       is well defined, continuous, and decays as $|s| \to \infty$ uniformly with respect to $\omega \in {\mathbb S}^1$. Then, instead of \eqref{24}, we have
       \bel{jun1}
       \lim_{q \to \infty} \lambda_q^{-1/2} {\rm Tr}\, \varphi(\lambda_q^{1/2} (H - \lambda_q)) = \frac{1}{2\pi} \int_{\re} \int_{{\mathbb S}^1} \varphi(B  \tilde{V}(s,\omega)) d\omega ds
     \ee
     (see \cite[Theorem 1.3]{prvb}). Note, in particular, that if $V$ satisfies \eqref{25} with $\rho > 1$, then \eqref{jun1} implies that ${\rm Tr}\, \varphi(\lambda_q^{1/2} (H - \lambda_q))$ is of order $\lambda_q^{1/2}$, while it follows from \eqref{24} that under the hypotheses of Theorem \ref{th1} ${\rm Tr}\, \varphi(\lambda_q^{\rho/2} (H - \lambda_q))$ is of order $\lambda_q$ as $q \to \infty$.
     \item Theorem \ref{th1} admits a similar semiclassical interpretation as \cite[Theorem 1.3]{prvb}. Namely, consider the classical Hamiltonian function
    \bel{gr1}
{\mathcal H}(\xi,x) = (\xi_1+ B x_2/2)^2+(\xi_2 -
Bx_1/2)^2, \; \xi  = (\xi_1,\xi_2) \in \rd, \; x  = (x_1,x_2) \in
\rd.
    \ee
 The projections onto the configuration space of the orbits
of the Hamiltonian flow of ${\mathcal H}$ are circles of radius
$\sqrt{E}/B$, where $E>0$ is the energy corresponding
to the orbit. The classical particles move around these circles
with  period $T_B = \pi/B$. These orbits are
parameterized by the energy $E>0$ and the center ${\bf
c}\in\re^2$ of the circle. Denote the path in the
configuration space corresponding to such an orbit by
$\gamma({\bf c},E,t)$, $t\in[0,T_B)$, and set
$$
{\rm Av}(V) ({\bf c},E)  = \frac{1}{T_B} \int_0^{T_B}
V(\gamma({\bf c},E,t))dt, \quad T_B=\pi/B.
$$
 It is easy to see that the r.h.s. of \eqref{24} can be
rewritten as
\bel{pa2}
\frac{1}{2\pi B} \int_{\rd} \varphi(B^{\rho} \vo(x)) dx = \frac{1}{2\pi}
\lim_{E\to\infty} \frac{1}{E}\int_{\re^2} \varphi(E^{\rho/2}\,
{\rm Av}(V) ({\bf c},E))  \, B\, d{\bf c}.
\ee
Given \eqref{pa2}, we can rewrite \eqref{24} as
\bel{pa3}
\lim_{q \to \infty} \frac{1}{\lambda_q} \tr \,\varphi(
\lambda_q^{\rho/2} (H-\lambda_q)) = \frac1{2\pi} \lim_{E\to\infty}
\frac1{E}\int_{\re^2} \varphi(E^{\rho/2}\, {\rm Av}(V)
({\bf c},E) ) \, B\, d{\bf c}.
\ee
Formula \eqref{pa3} agrees with
the so--called  ``averaging principle'' for systems
close to integrable ones, according to which  a good
approximation is obtained if one replaces the original
perturbation by its average  along the orbits of the free dynamics (see e.g. \cite[Section 52]{arn}).
\item Neither Theorem \ref{th1}, nor \cite[Theorem 1.3]{prvb}, treat the border-line case $\rho =1$, i.e. the case where $V$ is, say, asymptotically homogeneous of order $-1$. In this case the Radon transform of $V$ is not well defined, while the mean-value transform $\vo$ of $\V \in {\mathcal H}_{-1}^{\sharp}(\rd)$ generically is not bounded since it may have a logarithmic singularity at ${\mathbb S}^1$. Therefore, in the border-line case the asymptotic density of the eigenvalue clusters of $H$ should be different from both the short-range case $\rho > 1$ and the long-range case $\rho \in (0,1)$. Hopefully, we will consider in detail the border-line case in a future work.
\end{itemize}
The proof of Theorem \ref{th1} is contained on Section \ref{s5}.

\section{Auxiliary results} \label{s3}
\subsection{Weyl pseudodifferential operators} \label{ss31}

Let $d \geq 1$. Denote by  ${\mathcal S}(\re^d)$ the Schwartz
class, and by ${\mathcal S}'(\re^d)$ its dual class. If $f \in {\mathcal S}(\re^d)$, then
$$
\hat{f}(\xi) : = (2\pi)^{-d/2} \int_{\re^d} e^{-ix \cdot \xi} f(x) dx, \quad \xi \in \re^d,
$$
is the Fourier transform of $f$. Whenever necessary, we extend by duality the Fourier transform to ${\mathcal S}'(\re^d)$.\\
 Let $\Gamma(\re^{2d})$, $d\geq 1$,
denote the closure of $C_{\bf b}^\infty(\re^{2d})$ with respect to the norm
    $$
    \|s\|_{\Gamma(\re^{2d})} : = \sup_{\{\alpha , \beta \in {\mathbb
Z}_+^d \; | \; |\alpha|, |\beta| \leq  [\frac{d}{2}] + 1\}}
\sup_{(x,\xi) \in \re^{2d}} |\partial_x^{\alpha}
\partial_{\xi}^{\beta} s(x,\xi)| < \infty.
    $$

\begin{pr} \label{p9}
{\rm (\cite{cv, cor})} Assume that $s \in \Gamma(\re^{2d})$.
Then the operator $\Op(s)$ defined initially as a mapping from
 ${\mathcal S}(\re^d)$ into  ${\mathcal
S}'(\re^d)$ by
    \bel{gr52}
\left(\Op(s)u\right)(x) = (2\pi)^{-d} \int_{\re^d}
\int_{\re^d} s\left(\frac{x+x'}{2},\xi\right)e^{i(x-x')\cdot \xi}
u(x')
 dx'd\xi, \quad x \in \re^d,
    \ee
extends uniquely to an operator bounded in $L^2(\re^d)$, and there exists a constant $c_0$ independent of $s$ such that
$$
     \|\Op(s)\| \leq c_0\|s\|_{\Gamma(\re^{2d})}.
    $$
   \end{pr}
   \noindent
    The operator $\Op(s)$ is called {\em a pseudodifferential operator} ($\Psi$DO) {with Weyl symbol} $s$. Assume that $s \in {\mathcal S}'(\re^{2d})$,
    $\hat{s} \in L^1(\re^{2d})$. Then the operator defined in \eqref{gr52}
    extends to an operator bounded in $L^2(\re^{d})$, and we have
    \bel{g2}
\|\Op(s)\| \leq (2\pi)^{-d}\|\hat{s}\|_{L^1(\re^{2d})}
\ee
(see \cite[Lemma 18.6.1]{ho}).
    Assume now $s \in L^2(\re^{2d})$. Then, evidently, the operator defined in \eqref{gr52}
    extends to a Hilbert--Schmidt operator in $L^2(\re^{d})$, and we have
    \begin{equation} \label{prvb2}
\|\Op(s)\|_{2}^2 = (2\pi)^{-d}\int_{\R^{2d}} |s(x,\xi)|^2\,dx \,d\xi = (2\pi)^{-d}\int_{\R^{2d}} |\hat{s}(x,\xi)|^2\,dx \,d\xi.
\end{equation}
Let $X$ be a separable Hilbert space. Then $S_\infty(X)$ denotes the class of compact linear operators acting in $X$.
If $T \in S_\infty(X)$, then $\left\{s_j(T)\right\}_{j=1}^{{\rm rank}\,T}$ denotes the set of the non-zero singular numbers of $T$ enumerated in non-increasing order, and $S_\ell(X)$ is the Schatten--von Neumann class of order $\ell \in [1,\infty)$, i. e. the class of operators $T \in S_\infty(X)$ for which the norm
$\|T\|_\ell : = \left(\sum_{j=1}^{{\rm rank}\,T}s_j(T)^\ell\right)$
is finite. Thus, $S_1(X)$ is the trace class, and $S_2(X)$ is the Hilbert-Schmidt class. Similarly,
$S_{\ell, w}(X)$
denotes the {\em weak} Schatten--von Neumann class of order $\ell \in [1,\infty)$, i. e. the class of operators $T \in S_\infty(X)$ for which the quasinorm
$\|T\|_{\ell, w} : = \sup_{j} j^{1/\ell} s_j(T)$
is finite. Whenever appropriate, we omit $X$ in the notations $S_\ell(X)$ and $S_{\ell, w}(X)$.\\
Next, we recall that $u \in L^p_w(\re^d)$, $d \geq 1$, the weak Lebesgue space of order $p \in [1,\infty)$, if the quasinorm
$\|u\|_{L^p_w(\re^d)} : = \sup_{t>0}\,t\,\left| \left\{x \in \re^d \, | \, |u(x)| > t\right\}\right|^{1/p}$
is finite.
 Evidently, $u \in {\mathcal H}_{-\rho}^\sharp(\re^d)$, $\rho \in (0,d)$, implies $u \in L_w^{d/\rho}(\re^d)$.  \\
Interpolating between \eqref{g2} and \eqref{prvb2} (see \cite[Theorem 3.1]{bs}), we obtain the following
\begin{pr} \label{gp1}
Let $m \in (2,\infty)$, $m' : = m/(m-1)$. \\
{\rm (i)} Assume that $s \in {\mathcal S}'(\re^{2d})$, $\hat{s} \in
    L^{m'}(\re^{2d})$. Then $\Op(s) \in S_m(L^2(\re^d))$, and
    $$
    \|\Op(s)\|_m \leq (2\pi)^{-d\left(1-\frac{1}{m}\right)}\|\hat{s}\|_{L^{m'}(\re^{2d})}.
 $$
 {\rm (ii)} Assume that $s \in {\mathcal S}'(\re^{2d})$, $\hat{s} \in
    L^{m'}_w(\re^{2d})$. Then $\Op(s) \in S_{m, w}(L^2(\re^d))$, and
    $$
    \|\Op(s)\|_{m,w} \leq (2\pi)^{-d\left(1-\frac{1}{m}\right)}\|\hat{s}\|_{L^{m'}_w(\re^{2d})}.
 $$
\end{pr}

\subsection{Operators ${\rm Op}^{\rm w}(V_B * \Psi_q)$ and ${\rm Op}^{\rm w}(V_B * \delta_k)$} \label{ss32}
Introduce the harmonic oscillator
    $$
h : =  - \frac{d^2}{dx^2} + x^2,
   $$
self-adjoint in $L^2(\re)$.
It is well known
that the spectrum of $h$ is purely discrete and simple, and
consists of the eigenvalues $2q+1$, $q \in {\mathbb Z}_+$. Denote by $p_q$ the orthogonal projection onto
 ${\rm Ker}\,(h - (2q+1))$, $q \in {\mathbb Z}_+$.
 Set
    \bel{17a}
    \Psi_q(x,\xi) = \frac{(-1)^q}{\pi}
{\rm L}_q(2(x^2 + \xi^2)) e^{-(x^2 + \xi^2)}, \quad (x,\xi) \in
\rd, \quad q \in {\mathbb Z}_+,
    \ee
where
\bel{lagpol}
    {\rm L}_q(t) : = \frac{1}{q!} e^t \frac{d^q(t^q
e^{-t})}{dt^q} = \sum_{k=0}^q \binom{q}{k} \frac{(-t)^k}{k!},
\quad t \in \re, \ee
are the Laguerre polynomials. Then $2\pi \Psi_q$ is the Weyl symbol of the operator $p_q$. \\
Denote by $P_q$, $q \in \Z_+$, the orthogonal projection onto ${\rm Ker}\,(H_0-\lambda_q)$.
Set
    \bel{sof30}
    V_B(x)=V(-B^{-1/2}x_2,-B^{-1/2}x_1), \quad x = (x_1,x_2)\in\rd.
    \ee
    \begin{pr}\label{pr50}
{\rm (\cite[Corollary 2.13]{prvb})} There exists a unitary operator ${\mathcal U}_B : L^2(\rd) \to L^2(\rd)$ such that for each
 $V\in L^1(\rd) + L^\infty(\rd)$ and each $q \in {\mathbb Z}_+$ we have
    \bel{gr7}
{\mathcal U}_B^* P_q VP_q {\mathcal U}_B = p_q\otimes {\Op}
(V_B * \Psi_q).
    \ee
    \end{pr}
    \noindent
    For $k>0$, define the distribution $\delta_k \in {\mathcal S}'(\re^2)$  by
$$
\delta_k(\varphi) : =
\frac1{2\pi}\int_0^{2\pi}\varphi(k\cos\theta,k\sin\theta)d\theta,
\quad \varphi\in {\mathcal S}(\rd).
$$
\begin{pr} \label{p4}
Assume that $V \in {\mathcal S}_1^{-\rho}(\rd)$  with $\rho \in (0,\infty)$. Then the operator $\Op(V_B * \delta_k)$, $k > 0$,
is bounded and there exists a constant $c_1$ such that
\bel{33}
\|\Op(V_B * \delta_k)\| \leq c_1 \left\{
\begin{array} {l}
k^{-\rho} \quad \rm{if} \quad \rho \in (0,1),\\
k^{-1} \ln k \quad \rm{if} \quad \rho = 1,\\
k^{-1} \quad \rm{if} \quad \rho \in (1,\infty),
\end{array}
\right.
\quad k \in [2,\infty).
\ee
\end{pr}
\begin{proof}
Proposition \ref{p4} is an  extension of \cite[Lemma 3.2]{prvb} which concerned only the case $\rho > 1$.
By Proposition \ref{p9},
    \bel{d1}
    \|\Op(V_B * \delta_k)\| \leq c_0 \max_{\alpha \in \Z_+^2: 0 \leq |\alpha| \leq 2} \sup_{z \in \rd} |(D^\alpha V * \delta_k)(z)|.
    \ee
    By $V \in {\mathcal S}_1^{-\rho}(\rd)$, we have
    \bel{d2}
    |D^\alpha V(x)| \leq c_{1,\alpha} \langle x \rangle^{-|\alpha| - \rho} \leq c_{1,\alpha} \langle x \rangle^{ - \rho}, \quad x \in \rd, \quad \alpha \in \Z_+^2,
    \ee
    with constants $c_{1,\alpha}$ which may depend on $B$ but are independent of $x$. Now \eqref{d1} and \eqref{d2} imply
    \bel{d3}
     \|\Op(V_B * \delta_k)\| \leq c_1'   \sup_{z \in \rd} \langle \cdot \rangle^{-\rho} * \delta_k(z).
     \ee
     Note that the function  $\langle \cdot \rangle^{-\rho} * \delta_k$ is radially symmetric. Arguing as in the proof of \cite[Lemma 3.2]{prvb}, we get
     $$
     (\langle \cdot \rangle^{-\rho} * \delta_k)(z) = \frac{1}{2\pi} \int_0^{2\pi} ((k \cos{\theta} - |z|)^2 + k^2 \sin^2{\theta} + 1)^{-\rho/2} d\theta \leq
     $$
     \bel{d4}
      \frac{1}{2\pi} \int_0^{2\pi} (k^2 \sin^2{\theta} + 1)^{-\rho/2} d\theta = \frac{2}{\pi} \int_0^{\pi/2} (k^2 \sin^2{\theta} + 1)^{-\rho/2} d\theta \leq
      \int_0^1 (k^2 t^2 + 1)^{-\rho/2} dt = : I_\rho(k).
     \ee
     Elementary calculations yield
     \bel{d5}
I_\rho(k) =  \left\{
\begin{array} {l}
O(k^{-\rho}) \quad \rm{if} \quad \rho \in (0,1),\\
O(k^{-1} \ln k) \quad \rm{if} \quad \rho = 1,\\
O(k^{-1}) \quad \rm{if} \quad \rho \in (1,\infty),
\end{array}
\right.
\quad k \in [2,\infty).
\ee
Putting together \eqref{d3} -- \eqref{d5}, we obtain \eqref{33}.
\end{proof}

\begin{pr} \label{p5}
Assume that $V \in {\mathcal S}_1^{-\rho}(\rd)$  with $\rho \in (0,\infty)$. Then  $\Op(V_B * \Psi_q) - \Op(V_B * \delta_{\sqrt{2q+1}}) \in S_2$,
 and there exists a constant $c_2$ independent of $q$, such that
\bel{34}
\|\Op(V_B * \Psi_q) - \Op(V_B * \delta_{\sqrt{2q+1}})\|_2 \leq c_2 \lambda_q^{-3/4}, \quad q \in {\mathbb Z}_+.
\ee
\end{pr}
\begin{proof} Proposition \ref{p5} is an  extension of the second part of \cite[Lemma 3.1]{prvb} which concerned the case $V \in C_0^{\infty}(\rd)$. By \eqref{prvb2} we have
$$
\|\Op(V_B * \Psi_q) - \Op(V_B * \delta_{\sqrt{2q+1}})\|^2_2 = \frac{1}{2\pi} \int_{\rd} |(V_B * \Psi_q)(z) - (V_B * \delta_{\sqrt{2q+1}})(z)|^2 dz =
$$
\bel{d6}
 \frac{1}{2\pi} \int_{\rd} |(\widehat{V_B * \Psi_q})(\zeta) - (\widehat{V_B * \delta_{\sqrt{2q+1}}})(\zeta)|^2 d\zeta.
 \ee
An explicit calculation (see \cite[Eq. (3.9)]{prvb}) yields
\bel{d8}
(\widehat{V_B * \Psi_q})(\zeta) - (\widehat{V_B * \delta_{\sqrt{2q+1}}})(\zeta) =
\left({\rm L}_q(|\zeta|^2/2)e^{-|\zeta|^2/4} - J_0(\sqrt{2q+1} |\zeta|)\right) \widehat{V_B}(\zeta), \quad \zeta \in \rd,
\ee
where ${\rm L}_q$ is the Laguerre polynomial defined in \eqref{lagpol}, and $J_0$ is the Bessel function of zeroth order. Moreover,  there exists a constant
$\tilde{c}_2$ such that
\bel{d7}
\left|{\rm L}_q(r) e^{-r/2} - J_0(\sqrt{(4q+2)r})\right| \leq \tilde{c}_2 \left((q+1)^{-3/4} r^{5/4} + (q+1)^{-1} r^3\right), \quad q \in {\mathbb Z}_+, \quad r>0,
\ee
(see \cite[(Eq. (3.10)]{prvb} for the generic case $q \in \N$; if $q=0$, then \eqref{d7} follows from $|e^{-r/2} - J_0(\sqrt{r})| = O(r^2)|$, $r \in (0,1)$, and
$|e^{-r/2} - J_0(\sqrt{r})| = O(1)$, $r \geq 1$). Further,
$$
|\widehat{V_B}(\zeta)| = \left\{
\begin{array} {l}
O(|\zeta|^{-2+\rho}) \quad {\rm if} \quad \rho \in (0,2),\\
O(|\ln{|\zeta|}|)  \quad {\rm if} \quad \rho = 2,\\
O(1)  \quad {\rm if} \quad \rho > 2,
\end{array}
\right.
\quad |\zeta| \leq 1/2, \\
$$
and
$$
|\widehat{V_B}(\zeta)| = O(|\zeta|^{-N}), \quad |\zeta| > 1/2, \quad N>0,
$$
(see \cite[Chapter XII, Lemma 3.1]{taylor}). In particular, the functions
$|\zeta|^m \widehat{V_B}(\zeta)$, $\zeta \in \rd$, with $m > 1-\rho$ if $\rho \in (0,2)$ or with $m > -1$ if $\rho \geq 2$, are in $L^2(\rd)$. Combining
\eqref{d6}, \eqref{d8}, and \eqref{d7}, we get
$$
\|\Op(V_B * \Psi_q) - \Op(V_B * \delta_{\sqrt{2q+1}})\|^2_2 \leq
\frac{\tilde{c}_2^2}{\pi} \int_{\rd} \left((q+1)^{-3/2} |\zeta|^5 + (q+1)^{-2} |\zeta|^{12}\right)|\widehat{V_B}(\zeta)|^2 d\zeta,
$$
which yields \eqref{34}.
\end{proof}
\noindent
{\em Remark}: Estimate \eqref{34} could be interpreted as a manifestation of the equipartition of the eigenfunctions of the harmonic oscillator $h$, i.e. the appropriate weak convergence as $q \to \infty$ of the Wigner function $2\pi \Psi_q$ associated with the $q$th normalized eigenfunction of $H$, to the measure invariant with respect to the classical flow (see e.g. \cite{cdv, z, bdb} for related results concerning various ergodic quantum  systems).

\subsection{Norm estimates for Berezin--Toeplitz operators} \label{ss33}
\begin{pr} \label{p11}
 {\rm (i)} Let $V \in L^p(\rd)$, $p \in [1,\infty)$. Then for each $q \in {\mathbb Z}_+$ we have $P_q V P_q \in S_p$ and
\bel{31}
\|P_q V P_q\|_p^p  \leq \frac{B}{2\pi} \|V\|^p_{L^p(\rd)}.
\ee
{\rm (ii)} Let $V \in L^p_w(\rd)$, $p \in (1,\infty)$. Then for each $q \in {\mathbb Z}_+$ we have $P_q V P_q \in S_{p,w}$ and
\bel{g101}
\|P_q V P_q\|_{p,w}^p  \leq \frac{B}{2\pi} \|V\|^p_{L^p_w(\rd)}.
\ee
\end{pr}
\begin{proof}
Using the explicit expression for the integral kernel of $P_q$ (see e.g. \cite[Eq. (3.2)]{fr}), we easily obtain
    \bel{g102}
    \|P_q V P_q\|_1  \leq \frac{B}{2\pi} \|V\|_{L^1(\rd)}
    \ee
    with an equality if $V = \overline{V} \geq 0$. Moreover, evidently,
    \bel{g103}
    \|P_q V P_q\|  \leq  \|V\|_{L^\infty(\rd)}.
    \ee
    Interpolating between \eqref{g102} and \eqref{g103} (see \cite[Theorem 3.1]{bs}), we obtain \eqref{31} and \eqref{g101}.
    \end{proof}
    \noindent
{\em Remark}: The first part of Proposition \ref{p11} has been known since long ago (see (\cite[Lemma 5.1]{r0}, \cite[Lemma 3.1]{fr}).
\begin{follow} \label{f2}
Let $V \in L^p(\rd)$, $p \in [2,\infty)$. Then for each $q \in {\mathbb Z}_+$ we have $P_q V = (V P_q)^* \in S_p$, and
\bel{b1}
\|P_q V \|_p^p = \| V P_q\|_p^p   \leq \frac{B}{2\pi} \|V\|^p_{L^p(\rd)}.
\ee
\end{follow}
\begin{proof}
Estimates \eqref{b1} follow immediately from \eqref{31} since we have
$$
\|P_q V \|_p^p = \|P_q |V|^2 P_q \|_{p/2}^{p/2} \leq  \frac{B}{2\pi} \|V^2\|^{p/2}_{L^{p/2}(\rd)} = \frac{B}{2\pi} \|V\|^p_{L^p(\rd)}.
$$
\end{proof}
\begin{pr} \label{p2}
Assume that $V$ satisfies \eqref{25} with $\rho \in (0,\infty)$. Then there exists a constant $c_\infty$ such that
    \bel{32a}
    \|P_q V P_q\| \leq c_\infty \left\{
    \begin{array} {l}
    \lambda_q^{-\rho/2} \quad {\rm if} \quad \rho \in (0,1),\\
    \lambda_q^{-1/2}|\ln{\lambda_q}| \quad {\rm if} \quad \rho = 1,\\
    \lambda_q^{-1/2} \quad {\rm if} \quad \rho \in (1,\infty),
    \end{array}
    \right.
    \quad q \in {\mathbb Z}_+.
    \ee
    \end{pr}
\begin{proof}
An elementary variational argument implies that we may assume without loss of generality
that $V(x) = \langle x \rangle^{-\rho}$, $x \in \re^2$; then,  $V \in {\mathcal S}_1^{-\rho}(\rd)$. By Propositions \ref{pr50} and \ref{p5} we have
$$
 \|P_q V P_q\| =  \|\Op(V_B * \Psi_q)\| \leq
 $$
 $$
  \|\Op(V_B * \delta_{\sqrt{2q+1}})\| + \|\Op(V_B * \Psi_q) - \Op(V_B * \delta_{\sqrt{2q+1}})\| \leq
  $$
  $$
   \|\Op(V_B * \delta_{\sqrt{2q+1}})\| + \|\Op(V_B * \Psi_q) - \Op(V_B * \delta_{\sqrt{2q+1}})\|_2 \leq
   $$
   \bel{d9}
\|\Op(V_B * \delta_{\sqrt{2q+1}})\|  + c_2 \lambda_q^{-3/4}.
    \ee
    Now, \eqref{d9} and \eqref{33} yield immediately \eqref{32a}.
\end{proof}
\noindent
{\em Remark}: Estimates \eqref{32a} with $\rho \neq 1$ are sharp. For $\rho > 1$ this follows from the argument of
\cite{korpu} where the estimate $\|P_q V P_q\| \leq c_\infty \lambda_q ^{-1/2}$ was obtained for compactly supported $V$. Namely, if $\left\{\phi_{k,q}\right\}_{k=-q}^{\infty}$
is the so called {\em angular--momentum} orthonormal basis  of the Hilbert space $P_q\,L^2(\rd)$,
$q \in \Z_+$ (see e.g. \cite{rw}), and  $\one_R$ is the characteristic function of a disk of finite radius $R>0$, centered at the origin, then
$$
\liminf_{q \to \infty} \lambda_q^{1/2} \langle \one_R \phi_{0,q} , \phi_{0,q}\rangle_{L^2(\rd)} > 0,
$$
which implies the sharpness of estimates \eqref{32a} with $\rho>1$.
Similarly, if $\rho \in (0,1)$, we can show that
$$
\liminf_{q \to \infty} \lambda_q^{\rho/2} \langle \langle \cdot \rangle^{-\rho} \phi_{-q,q} , \phi_{-q,q}\rangle_{L^2(\rd)} > 0,
$$
which entails the sharpness of estimates \eqref{32a} with $\rho  \in (0,1)$.
We do not know whether estimate \eqref{32a} with $\rho = 1$ is sharp, but it is sufficient for our purposes.

\begin{follow} \label{f1}
Assume that $V$ satisfies \eqref{25} with $\rho \in (0,1)$.
 Then for each $\ell > 2/\rho$
there exists a constant $c_{\ell}$ such that
    \bel{32}
    \|P_q V P_q\|_\ell \leq c_{\ell} \lambda_q^{\frac{1}{\ell}-\frac{\rho}{2}}\,(1 + |\ln{\lambda_q}|)^{1/\ell}, \quad q \in {\mathbb Z}_+.
    \ee
    \end{follow}
\begin{proof}
 Similarly to the proof of Proposition \ref{p2} above,
an elementary variational argument shows that we may assume without loss of generality
that $V(x) = \langle x \rangle^{-\rho}$, $x \in \re^2$. Note also that in the proof of estimate \eqref{32} we may assume that $q$ is large enough since for any fixed $q$ it follows  from \eqref{31}. \\
For brevity, set $T_q : = P_q V P_q$, $q \in \Z_+$; by \cite{rw}, ${\rm rank}\,T_q = \infty$. By \eqref{g101} with $p = 2/\rho$, there exists a constant ${\mathcal C}$ such that
    \bel{g105}
s_j(T_q) \leq {\mathcal C} j^{-\rho/2}, \quad j \in {\mathbb N}, \quad q \in \Z_+.
    \ee
On the other hand, \eqref{32a} implies
\bel{g106}
s_1(T_q) \leq c_\infty \lambda_q^{-\rho/2},  \quad q \in \Z_+.
    \ee
    Fix $\ell > 2/\rho$. By \eqref{g105} -- \eqref{g106}, for any $N \in \N$, we have
    $$
    \|T_q\|_\ell^\ell = \sum_{j=1}^{\infty} s_j(T_q)^\ell = \sum_{j=1}^{N} s_j(T_q)^\ell + \sum_{j=N+1}^{\infty} s_j(T_q)^\ell \leq
    $$
    $$
    s_1(T_q)^{\ell - \frac{2}{\rho}}\sum_{j=1}^{N} s_j(T_q)^{\frac{2}{\rho}} + {\mathcal C}^\ell \sum_{j=N+1}^{\infty} j^{-\frac{\ell \rho}{2}} \leq
    c_\infty^{\ell - \frac{2}{\rho}} {\mathcal C}^{2/\rho} \lambda_q^{1-\frac{\ell \rho}{2}} \sum_{j=1}^{N} j^{-1} + {\mathcal C}^\ell \sum_{j=N+1}^{\infty} j^{-\frac{\ell \rho}{2}} \leq
    $$
    $$
    {\rm const.} \left(\lambda_q^{1-\frac{\ell \rho}{2}}(1 + \ln{N}) + N^{1-\frac{\ell \rho}{2}}\right)
    $$
    with a constant independent of $N$ and $q$. Assuming that $q$ is large enough, and choosing $N$ equal to the integer part of $\lambda_q$, we obtain \eqref{32}.
   \end{proof}
   \noindent
   {\em Remark}:
   %Propositions \ref{p10} and \ref{p11a} imply that if $\rho \in (0,1)$, $\ell > 2/\rho$, and $V(x) = \langle x \rangle^{-\rho}$, $x \in \rd$, then
   %$$
   %\|\lambda_q^{\rho/2} P_q V P_q\|_\ell^\ell = {\rm Tr}\,(\lambda_q^{\rho/2} P_q V P_q)^\ell =
  % $$
  % $$
  % \frac{B^{\ell \rho - 1}}{2\pi} \int_{\rd} \vo(x)^\ell dx \lambda_q(1 + o(1)), \quad q \to \infty,
  % $$
  % where $\vo$ is the mean-value transform of the function $\V(x) =|x|^{-\rho}$, $x \in \rd \setminus \{0\}$.
   Estimate \eqref{32}  should be regarded as an {\em a priori} estimate which is  sufficient for our purposes.

\subsection{Proof of Proposition \ref{p1}} \label{ss41}
Given estimate \eqref{32a} with $\rho \in (0,1)$, the proof of Proposition \ref{p1} is analogous to the one of \cite[Proposition 1.1]{prvb};
we include it just for the convenience of the reader. \\
In order to prove \eqref{23} it suffices to show that there exist $\tilde{C}>0$ and $s_0 \in \N$ such that $s \geq s_0$ implies
    \bel{f1a}
    \sigma(H) \cap [\lambda_s-B, \lambda_s + B] \subset \left(\lambda_s  - \tilde{C} \lambda_s^{-\rho/2}, \lambda_s  + \tilde{C} \lambda_s^{-\rho/2}\right).
    \ee
    Set $R_0(z) = (H_0 - z)^{-1}$, $z \in \C \setminus \sigma(H_0)$. By the Birman--Schwinger principle,
    $\lambda \in  \re \setminus \sigma(H_0)$ is an eigenvalue of $H$ if and only if $-1$ is an eigenvalue of $|V|^{1/2} R_0(\lambda) V^{1/2}$ where
    $$
    V^{1/2}(x) : = \left\{
    \begin{array} {l}
    |V(x)|^{1/2} \, {\rm sign}\, V(x) \quad {\rm if} \quad V(x) \neq 0,\\
    0  \quad {\rm if} \quad V(x) = 0.
    \end{array}
    \right.
    $$
    Hence, in order to prove \eqref{f1a}, it suffices to show that for some $\tilde{C} > 0$ and $s_0 \in \N$, the inequalities $s \geq s_0$ and
    \bel{f0}
    \tilde{C} \lambda_s^{-\rho/2} < |\lambda_s - \lambda| \leq B
    \ee
    imply
    \bel{f5}
    \||V|^{1/2} R_0(\lambda) |V|^{1/2}\| < 1.
    \ee
    Pick $m \in \N$ such that
    $\|V\|_{L^\infty(\rd)} \leq \lambda_m/2$.
     For $s \geq m$ write
    $$
    R_0(\lambda) = \sum_{k=s-m}^{s+m} (\lambda_k-\lambda)^{-1} P_k + \tilde{R}_0(\lambda; s,m).
    $$
    Then
     \bel{f6}
     \||V|^{1/2} R_0(\lambda) V^{1/2}\| \leq
      \sum_{k=s-m}^{s+m} |\lambda_k-\lambda|^{-1} \|P_k |V| P_k\| + \| |V|^{1/2} \tilde{R}_0(\lambda; s,m) |V|^{1/2}\|.
     \ee
     By the choice of $m$, we have
     \bel{f3}
     \| |V|^{1/2} \tilde{R}_0(\lambda; s,m) |V|^{1/2}\| < \frac{1}{2}.
     \ee
     On the other hand, by \eqref{32a} with $\rho \in (0,1)$, we have
     $$
     \sum_{k=s-m}^{s+m}  |\lambda_k-\lambda|^{-1} \|P_k |V| P_k\| \leq c_\infty \lambda_{s-m}^{-\rho/2} (2m+1) |\lambda_s - \lambda|^{-1}
     $$
     which implies
     \bel{f7}
     \sum_{k=s-m}^{s+m}  |\lambda_k-\lambda|^{-1} \|P_k |V| P_k\| < \frac{1}{2},
     \ee
     provided that the first inequality in \eqref{f0} holds with appropriate $\tilde{C}$. Now, \eqref{f5} follows from \eqref{f6}, \eqref{f3}, and \eqref{f7}.\\

\noindent
In the proof of Theorem \ref{th1}, we will need also the following
\begin{pr} \label{p6}
Assume that $V$ satisfies \eqref{25} with $\rho \in (0,1)$. Then there exists a constant $C' > 0$ such that for each $q \in {\mathbb Z}_+$ we have
\bel{41}
\sigma\left((I-P_q)H(I-P_q)_{|(I-P_q){\rm Dom}(H_0)}\right) \subset
\bigcup_{s \in {\mathbb Z}_+\setminus \{q\}} \left(\lambda_s - C'\lambda_s^{-\rho/2}, \lambda_s + C'\lambda_s^{-\rho/2}\right).
\ee
\end{pr}
\noindent
 The proof of Proposition \ref{p6}  is quite the same as that of Proposition \ref{p1}, so that we omit the details.

\section{Proof of Theorem \ref{th1}} \label{s5}
\subsection{Passing from $V$ to its Weinstein average $\langle V \rangle$} \label{ss51}

Assume that $V \in L^{\infty}(\rd)$ and set
    \bel{aug1}
\langle V \rangle : = \sum_{s \in {\mathbb Z}_+} P_s V P_s
    \ee
where, a priori, the series converges strongly; this is the case if, for instance, $V = 1$ identically. If
    \bel{aug2}
\lim_{q \to \infty}\|P_q VP_q\| = 0
    \ee
 which, by Proposition \ref{p2}, is the case if $V$ satisfies \eqref{25} with $\rho > 0$, then the series in \eqref{aug1}
converges in norm. In order to check this, it suffices to show that $\left\{\langle V \rangle_q\right\}_{q \in {\mathbb Z}_+}$ with
$\langle V \rangle_q : = \sum_{s=0}^q P_s V P_s$, $q \in {\mathbb Z}_+$, is a Cauchy sequence in the uniform operator topology. By $P_j P_s = 0$ for $j \neq s$, we have
$$
\|\langle V \rangle_{q + m} - \langle V \rangle_q\| \leq \sup_{j \geq q+1} \|P_j V P_j\|, \quad q \in \Z_+, \quad m \in \N,
$$
which combined with \eqref{aug2},
%\eqref{32a},
implies the required property of the sequence $\left\{\langle V \rangle_q\right\}_{q \in {\mathbb Z}_+}$. Since
$$
\langle V \rangle = \frac{B}{\pi} \int_0^{\pi/B} e^{-it H_0} V e^{it H_0}\,dt,
$$
 we call $\langle V \rangle $ the  Weinstein average of $V$ (see \cite{wein}). Set $\langle H \rangle : = H_0 + \langle V \rangle$.

\begin{pr} \label{p7}
Under the hypotheses of Theorem \ref{th1} we have
\bel{51}
{\rm Tr}\, \varphi(\lambda_q^{\rho/2} (H - \lambda_q)) = {\rm Tr}\, \varphi(\lambda_q^{\rho/2} (\langle H \rangle  - \lambda_q)) + o(\lambda_q),
\quad q \to \infty.
     \ee
     for each $\varphi \in C_0^{\infty}(\re\setminus\{0\})$.
     \end{pr}
     \begin{proof}
     First, let us write the difference of the traces in \eqref{51} according to the Helffer--Sj\"ostrand formula (see the original
     works \cite{dyn, helsj},
     or the monographs \cite[Section 2.2]{davies}, \cite[Chapter 8]{dsj}). Let $\varphi \in C_0^{\infty}(\re\setminus\{0\})$, and let
     $\tilde{\varphi} \in C_0^{\infty}(\rd)$ be an almost analytic continuation of $\varphi$ which satisfies
     \bel{a1}
     {\rm supp}\,\tilde{\varphi} \subset ((a_-, b_-) \cup (a_+, b_+)) \times (-c,c)
     \ee
     with $-\infty < a_- < b_- < 0 < a_+ < b_+ < \infty$, and $0 < c < \infty$, as well as
     \bel{a2}
     |\psi(x,y)| \leq C_N |y|^N, \quad (x,y) \in \rd, \quad N>0,
     \ee
     where
     $\psi : = \frac{1}{2} \left(\frac{\partial \tilde{\varphi}}{\partial x} + i \frac{\partial \tilde{\varphi}}{\partial y}\right)$.
     For $(x,y) \in \rd$ set $z = x+ iy$ and
     \bel{a2a}
     \psi_q(x,y) : = \lambda_q^{\rho/2} \psi(\lambda_q^{\rho/2}(x-\lambda_q), \lambda_q^{\rho/2}y), \quad  q \in {\mathbb Z}_+.
    \ee
     Then the Helffer--Sj\"ostrand formula yields
     $$
     \varphi(\lambda_q^{\rho/2} (H - \lambda_q)) = \frac{1}{\pi} \int_{\rd} \psi(x,y) (\lambda^{\rho/2}(H -\lambda_q) - z)^{-1} dx dy =
      \frac{1}{\pi} \int_{\rd} \psi_q(x,y) (H - z)^{-1} dx dy.
      $$
      Similarly,
      $$
     \varphi(\lambda_q^{\rho/2} (\langle H \rangle - \lambda_q)) =
     \frac{1}{\pi} \int_{\rd} \psi_q(x,y) (\langle H \rangle - z)^{-1} dx dy.
      $$
      Further, let $\lo$ be the smallest integer (strictly) greater than $2/\rho$. Write the iterated resolvent identity
      \bel{a5}
      (H-z)^{-1} = \sum_{s=0}^{\lo -1} (-1)^s ((H_0 - z)^{-1}V)^s (H_0-z)^{-1} + (-1)^{\lo} ((H_0 - z)^{-1}V)^{\lo} (H-z)^{-1}.
      \ee
      In the sequel, assume that $q \in \Z_+$ is so large that $-2B < a_- \lambda_q^{-\rho/2}$ and $b_+ \lambda_q^{-\rho/2} < 2B$ (see \eqref{a1} and \eqref{a2a}).
      Then the sum on the r.h.s of \eqref{a5} is holomorphic on the support of $\psi_q$. Therefore,
      $$
      \varphi(\lambda_q^{\rho/2} (H - \lambda_q)) = \frac{(-1)^{\lo}}{\pi} \int_{\rd} \psi_q(x,y) ((H_0 - z)^{-1}V)^{\lo} (H-z)^{-1} dx dy.
      $$
      Similarly,
      $$
      \varphi(\lambda_q^{\rho/2} (\langle H \rangle - \lambda_q)) =
      \frac{(-1)^{\lo}}{\pi} \int_{\rd} \psi_q(x,y) ((H_0 - z)^{-1}\langle V \rangle)^{\lo} (\langle H \rangle-z)^{-1} dx dy.
      $$
      Thus we get
      $$
      (-1)^{\lo} \pi \,{\rm Tr}\,\left(\varphi(\lambda_q^{\rho/2} (H - \lambda_q)) - \varphi(\lambda_q^{\rho/2} (\langle H \rangle - \lambda_q))\right) =
      $$
      $$
      \int_{\rd} \psi_q(x,y)\left(((H_0 - z)^{-1}V)^{\lo} (H-z)^{-1}- ((H_0 - z)^{-1}\langle V \rangle)^{\lo} (\langle H \rangle-z)^{-1}\right) dx dy =
      $$
      \bel{a8}
      \sum_{j=1,2,3} T_j(q)
      \ee
      where
      $$
      T_1(q) : = {\rm Tr}\,\int_{\rd} \psi_q(x,y) \left(((H_0 - z)^{-1}V)^{\lo} - (\lambda_q -z)^{-\lo}(P_q V)^{\lo}\right) (H-z)^{-1} dx dy,
      $$
      $$
      T_2(q) : = {\rm Tr}\,\int_{\rd} \psi_q(x,y) \left((\lambda_q -z)^{-\lo}(P_q \langle V \rangle)^{\lo} - ((H_0 - z)^{-1}\langle V \rangle)^{\lo}\right) ( \langle H \rangle-z)^{-1} dx dy,
      $$
      $$
      T_3(q) : = {\rm Tr}\,\int_{\rd} \psi_q(x,y) (\lambda_q -z)^{-\lo}(P_q V)^{\lo}\left((H-z)^{-1}-P_q(\langle H \rangle-z)^{-1}\right) dx dy.
      $$
      Writing $T_3(q)$, we have taken into account that
      $(P_q \langle V \rangle)^{\lo} = (P_q  V P_q)^{\lo} = (P_q  V)^{\lo} P_q$.
      Hence, in order to prove \eqref{51} it suffices to show that
      $$
      T_j(q) = o(\lambda_q), \quad j=1,2,3, \quad q \to \infty.
      $$
       To this end we need some preliminary estimates.
      Namely, we will show that if $V \in L^p(\rd)$, $p \in [2,\infty)$,  $(x,y) \in {\rm supp}\,\psi_q$, and $q \in {\mathbb Z}_+$,
      then $(I-P_q) (H_0-z)^{-1}V \in S_p$,
      and there exists a constant $c_p$ such that
      \bel{a9}
      \sup_{(x,y) \in {\rm supp}\,\psi_q} \, \|(I-P_q) (H_0-z)^{-1}V\|_p \leq c_p \|V\|_{L^p(\rd)}
      \ee
      for sufficiently large $q$. Assume at first $V \in L^{\infty}(\rd)$. Then, evidently,
      \bel{a10}
      \|(I-P_q) (H_0-z)^{-1}V\| \leq \sup_{s \in \Z_+\setminus \{q\}} |\lambda_s - z|^{-1} \|V\|_{L^{\infty}(\rd)}.
      \ee
      Note that if $(x,y) \in {\rm supp}\,\psi_q$, then
      \bel{a12}
      |x-\lambda_q| = O(\lambda_q^{-\rho/2}), \quad q \to \infty,
      \ee
      (see \eqref{a1} and  \eqref{a2a}). Hence, for $q$ large enough we have
      \bel{a11}
      \sup_{s \in \Z_+\setminus \{q\}} |\lambda_s - z|^{-1} \leq B^{-1}, \quad  (x,y) \in {\rm supp}\,\psi_q.
      \ee
      Assume now that $V \in L^2(\rd)$. Then we have
      \bel{a13}
       \|(I-P_q) (H_0-z)^{-1}V\|_2^2 = \sum_{s \in \Z_+\setminus \{q\}} |\lambda_s - z|^{-2} {\rm Tr} P_s |V|^2 P_s =
       \frac{B}{2\pi} \sum_{s \in \Z_+\setminus \{q\}} |\lambda_s - z|^{-2}\|V\|^2_{L^2(\rd)}
       \ee
       (see \eqref{g102}). Taking into account again \eqref{a12}, we find that for $q$ large enough we have
       \bel{a14}
       \sum_{s \in \Z_+\setminus \{q\}} |\lambda_s - z|^{-2} \leq \frac{\pi^2}{3B^2}, \quad  (x,y) \in {\rm supp}\,\psi_q.
       \ee
       Interpolating between \eqref{a10} and \eqref{a13} (see \cite[Theorem 3.1]{bs}), and bearing in mind \eqref{a11} and \eqref{a14}, we obtain \eqref{a9}.
       Further, since
       $$
       (H_0-z)^{-1}V = (\lambda_q - z)^{-1} P_qV + (I-P_q) (H_0-z)^{-1}V,
       $$
       elementary combinatorial arguments yield the estimate
       \bel{a15}
       |T_1(q)| \leq \sum_{m=0}^{\lo - 1} \binom{\lo}{m} T_{1, m}(q)
       \ee
       where
       \bel{a16}
       T_{1, m}(q) : = \int_{\rd} |\psi_q(x,y)| \|(I-P_q)(H_0-z)^{-1}V\|_{\lo}^{\lo - m}|\lambda_q-z|^{-m} \|P_q V\|_{\lo}^m \|(H-z)^{-1}\| dx dy.
       \ee
       Our next goal is to show that
       \bel{a19}
       T_{1, m}(q) = o(\lambda_q), \quad q \to \infty, \quad m = 0, \ldots, \lo - 1.
       \ee
       To this end we apply:
       \begin{itemize}
       \item estimate \eqref{a2} with $N=2$ in order to get
       $$
       \sup_{(x,y) \in {\rm supp}\,\psi_q} |\psi_q(x,y)| \leq C_2\lambda_q^{3\rho/2} y^2;
       $$
       \item estimate \eqref{a9} in order to handle $ \|(I-P_q)(H_0-z)^{-1}V\|_{\lo}$;
       \item the fact that, due to \eqref{a1} and \eqref{a2a}, we have
       \bel{a23}
        \sup_{(x,y) \in {\rm supp}\,\psi_q} |\lambda_q - z|^{-1} = O(\lambda_q^{\rho/2});
        \ee
       \item estimate \eqref{b1} in order to handle $\|P_q V\|_{\lo}$;
       \item the standard resolvent estimate $\|(H-z)^{-1}\| \leq |y|^{-1}$;
       \item the elementary estimate
       \bel{a25}
       \int_{{\rm supp}\,\psi_q} |y| dxdy = O(\lambda_q^{-3\rho/2}).
       \ee
        \end{itemize}
        As a result, we obtain
        \bel{a17}
        T_{1,m}(q ) \leq {\rm const.} \, \lambda_q^{m\rho/2}, \quad m = 0, \ldots, \lo - 1,
        \ee
        with a constant independent of $q$. We have  $\lambda_q^{m\rho/2} = o(\lambda_q)$ as $q \to \infty$ in all the cases except the one where $2/\rho$
        is an integer, and $m = \lo - 1 = 2/\rho$. In this exceptional case however we have $m\geq3$ and in all the terms of
        $$
        \left((H_0-z)^{-1} V\right)^{\lo} - (\lambda_q -z)^{-\lo}(P_q V)^{\lo}
        $$
        which contain $m = \lo - 1$ factors of the type $(\lambda_q -z)^{-1} P_q V$, at least two
        of these factors are neighbours. Therefore, in this exceptional case we can replace $\|P_q V\|_{\lo}^{m}$ by
        $\|P_q V\|_{\lo}^{m-1} \|P_q V P_q\|_{\lo}$ in \eqref{a16}, apply \eqref{32} with $\ell = \lo = m+ 1$, and obtain
        \bel{a18}
        T_{1,m}(q ) = O\left(\lambda_q^{1 + \frac{1}{m+1} - \frac{1}{m}} (\ln{\lambda_q})^\frac{1}{m+1}\right)= o(\lambda_q), \quad q \to \infty.
        \ee
        Now, \eqref{a17} and \eqref{a18} entail \eqref{a19},
         which combined with \eqref{a15} implies
        \bel{a42}
        |T_1(q)| = o(\lambda_q), \quad q \to \infty.
        \ee
        Similarly, we get
        \bel{a20}
        |T_2(q)| = o(\lambda_q), \quad q \to \infty.
        \ee
        Let us now turn to $T_3(q)$. First, note that due to the cyclicity of the trace, we have
        $$
        T_3(q) =
        $$
        $$
        {\rm Tr}\,\int_{\rd} \psi_q(x,y) (\lambda_q -z)^{-\lo} (P_q V)^{\lo} \left((H-z)^{-1} P_q - P_q( P_q V P_q + \lambda_q - z)^{-1}P_q\right) dxdy.
        $$
        Next, we need the Schur--Feshbach  formula (see the original works  \cite{schur, f}, or a contemporary exposition available, for instance,
        in \cite[Appendix]{bhs}). According to this formula,
        $$
        (H-z)^{-1} =
        $$
        $$
        P_q R_{\parallel}(z) P_q - P_q R_{\parallel}(z) V(I-P_q) R_{\perp}(z) - (I-P_q) R_{\perp}(z) V P_q R_{\parallel}(z) +
        $$
        \bel{a100}
        (I-P_q)\left(R_{\perp}(z) + R_{\perp}(z)V P_q R_{\parallel}(z) V(I-P_q) R_{\perp}(z)\right)
        \ee
         where $R_{\perp}(z)$ is the inverse of the operator $(I-P_q)(H-z)(I-P_q)$ defined on $(I-P_q){\rm Dom}\,H_0$,
        and considered as an operator in the Hilbert space $(I-P_q) L^2(\rd)$, while $R_{\parallel}(z)$ is the inverse of the operator
        $$
        P_q V P_q - P_q V (I-P_q) R_{\perp}(z) V P_q + \lambda_q - z
        $$
        considered as an operator in the Hilbert space $P_q L^2(\rd)$.
        Applying \eqref{a100} and the resolvent identity, we obtain
        $$
        (H-z)^{-1} P_q - P_q(\lambda_q + P_q V P_q - z)^{-1}P_q =
        $$
        $$
        R_{\parallel} P_q V (I-P_q) R_{\perp}(z) V P_q (P_q V P_q + \lambda_q - z)^{-1} - (I-P_q) R_{\perp}(z) V R_{\parallel}(z) P_q
        $$
        Thus,
        \bel{a33}
        T_3(q) = T_{3,1}(q) + T_{3,2}(q)
        \ee
        where
        $$
        T_{3,1}(q) : =
        $$
        $$
        {\rm Tr}\,\int_{\rd}\frac{\psi_q(x,y)}{(\lambda_q -z)^{\lo}} (P_q V P_q)^{\lo} R_{\parallel}(z) P_q V (I-P_q) R_{\perp}(z) V P_q (P_q V P_q + \lambda_q - z)^{-1}dxdy,
         $$
         and
         $$
        T_{3,2}(q) : =
        - {\rm Tr}\,\int_{\rd} \psi_q(x,y) (\lambda_q -z)^{-\lo} (P_q V P_q)^{\lo - 1} P_q V
        (I-P_q) R_{\perp}(z) V P_q R_{\parallel}(z) P_q dxdy.
         $$
         We have
         $$
         |T_{3,1}(q)| \leq
         $$
         \bel{a21}
         \int_{\rd} \frac{|\psi_q(x,y)|}{|\lambda_q -z|^{\lo}} \|P_q V P_q\|_{\lo}^{\lo} \|R_{\parallel}(z)\| \|P_q V\| \|V P_q \| \| R_{\perp}(z)\|
          \|(P_q V P_q + \lambda_q - z)^{-1}\| dxdy.
          \ee
          In order to show that
          \bel{a31}
          |T_{3,1}(q)| = o(\lambda_q), \quad q \to \infty,
          \ee
          we apply:
          \begin{itemize}
          \item estimate \eqref{a2} with $N=3$ in order to get
       $$
       \sup_{(x,y) \in {\rm supp}\,\psi_q} |\psi_q(x,y)| \leq C_3\lambda_q^{2\rho} |y|^3;
       $$
       \item estimate \eqref{a23} in order to handle $|\lambda_q - z|^{-\lo}$;
       \item estimate \eqref{32} with $\ell = \lo$ in order to handle $\|P_q V P_q\|_{\lo}$;
       \item the standard resolvent estimates
       \bel{a28}
        \|R_{\parallel}(z)\| \leq |y|^{-1}, \quad \|(P_q V P_q + \lambda_q - z)^{-1}\| \leq |y|^{-1};
        \ee

       \item estimate \eqref{32a} in order to conclude that
       \bel{a29}
       \|P_q V\| \|V P_q\| = \|P_q V^2 P_q\| \leq c_\infty \left\{
    \begin{array} {l}
    \lambda_q^{-\rho} \quad {\rm if} \quad \rho \in \left(0,\frac{1}{2}\right),\\
    \lambda_q^{-1/2}|\ln{\lambda_q}| \quad {\rm if} \quad \rho = \frac{1}{2},\\
    \lambda_q^{-1/2} \quad {\rm if} \quad \rho \in \left(\frac{1}{2},1\right),
    \end{array}
    \right.
    \quad q \in {\mathbb Z}_+;
       \ee
        \item the elementary estimate \eqref{a25};
        \item Proposition \ref{p6} in order to deduce the estimate
         \bel{a30}
         \sup_{(x,y) \in {\rm supp}\,\psi_q} \|R_{\perp}(z)\| = O(1), \quad q \to \infty;
         \ee
          \end{itemize}
          As a result, we obtain
          \bel{a40}
          |T_{3,1}(q)| = O\left(\Phi_{1,\rho}(\lambda_q)\right), \quad q \to \infty,
          \ee
          where
          $$
          \Phi_{1,\rho}(t) : = \left\{
    \begin{array} {l}
    t^{1-\rho/2} |\ln{t}| \quad {\rm if} \quad \rho \in \left(0,\frac{1}{2}\right),\\
    t^{3/4}(\ln{t})^2 \quad {\rm if} \quad \rho = \frac{1}{2},\\
    t^{(1+\rho)/2} |\ln{t}|\quad {\rm if} \quad \rho \in \left(\frac{1}{2},1\right),
    \end{array}
    \right.
    \quad t > 0.
    $$
    Now, \eqref{a40}  implies \eqref{a31}.
    Finally, we have
    $$
         |T_{3,2}(q)| \leq
         \int_{\rd} |\psi_q(x,y)| |\lambda_q -z|^{-\lo} \|P_q V P_q\|_{\lo}^{\lo - 1}  \|P_q V \|_{\lo}\|R_{\perp}(z)\| \|P_q V\|  \|R_{\parallel}(z)\|
           dxdy
          $$
          In order to show that
          \bel{a32}
          |T_{3,2}(q)| = o(\lambda_q), \quad q \to \infty,
          \ee
          we apply \eqref{a2} with $N=2$, \eqref{a23}, \eqref{32}, \eqref{b1}, the first estimate in \eqref{a28}, \eqref{a30}, \eqref{a25},
           and \eqref{a29}. Thus we obtain
           \bel{a41}
           |T_{3,2}(q)| = O\left(\lambda_q^{\frac{(\lo - 1)}{\lo} + \frac{\rho}{2}} (\ln{\lambda_q})^{\frac{(\lo - 1)}{\lo}} \Phi_{2,\rho}(\lambda_q)\right), \quad q \to \infty,
           \ee
           where
           $$ \Phi_{2,\rho}(t) : = \left\{
    \begin{array} {l}
    t^{-\rho/2} \quad {\rm if} \quad \rho \in \left(0,\frac{1}{2}\right),\\
    t^{-1/4}|\ln{t}|^{1/2} \quad {\rm if} \quad \rho = \frac{1}{2},\\
    t^{-1/4} \quad {\rm if} \quad \rho \in \left(\frac{1}{2},1\right),
    \end{array}
    \right. \quad t > 0,
           $$
           and \eqref{a32} follows from \eqref{a41}.\\
           \noindent
             Now
           the combination of \eqref{a8}, \eqref{a42}, \eqref{a20}, \eqref{a33}, \eqref{a31}, and \eqref{a32} yields \eqref{51}.
          \end{proof}

\subsection{Passing to individual Berezin--Toeplitz operators} \label{ss52}
\begin{pr} \label{p8}
Assume  the hypotheses of Theorem \ref{th1}. Then for each $\varphi \in C_0^{\infty}(\re\setminus\{0\})$ there exists $q_0 \in {\mathbb Z}_+$ such that
\bel{52}
 {\rm Tr}\, \varphi(\lambda_q^{\rho/2} (\langle H \rangle  - \lambda_q)) =  {\rm Tr}\, \varphi(\lambda_q^{\rho/2} P_q V P_q)
\ee
     for $q \geq q_0$.
     \end{pr}
     \begin{proof}
     We have
     \bel{c1}
      {\rm Tr}\, \varphi(\lambda_q^{\rho/2} (\langle H \rangle  - \lambda_q)) =  \sum_{s \in \Z_+}{\rm Tr}\, \varphi(\lambda_q^{\rho/2}(\lambda_s - \lambda_q + P_s V P_s))
     \ee
      Due to the presence of the  factor $\lambda_q^{\rho/2}$ in the traces ${\rm Tr}\, \varphi(\lambda_q^{\rho/2}(\lambda_s - \lambda_q + P_s V P_s))$, $s \in \Z_+$, it suffices to show that there exists $q_0$ such that for $q \geq q_0$ the operators
     \bel{c2}
     \lambda_s - \lambda_q + P_s V P_s = 2B(s-q) + P_s V P_s, \quad s \neq q,
     \ee
     are invertible, and
     \bel{c3}
     \sup_{q \geq q_0} \sup_{s \in \Z_+ \setminus \{q\}} \|( \lambda_s - \lambda_q + P_s V P_s)^{-1}\| < \infty.
     \ee
     Since $\|P_s V P_s\| \leq \|V\|_{L^{\infty}(\rd)}$, $s \in \Z_+$, there exists $m \in \N$ such that the operators in \eqref{c2} with $|s-q| > m$ are invertible, and
     \bel{c4}
     \sup_{q \in \Z_+} \, \sup_{s \in \Z_+ : |s-q| > m} \|( \lambda_s - \lambda_q + P_s V P_s)^{-1}\| < \infty.
     \ee
     On the other hand, Proposition \ref{p2} implies that for any fixed $j \in \Z$ we have
     $$
     \lim_{q \to \infty} \|P_{q+j} V P_{q+j}\| = 0.
     $$
     Therefore, there exists $q_0 \in \Z_+$ such that
     the operators in \eqref{c2} with $|s-q| \leq m$ are invertible for $q \geq q_0$, and
     \bel{c5}
     \sup_{q \geq q_0} \, \max_{s \in \Z_+ \setminus \{q\}: |s-q| \leq m} \|( \lambda_s - \lambda_q + P_s V P_s)^{-1}\| < \infty.
     \ee
     Putting together \eqref{c4} and \eqref{c5}, we obtain \eqref{c3}, and hence \eqref{52}.
\end{proof}

\subsection{Passing from $P_q V P_q$ to $\Op(\V_B * \delta_{\sqrt{2q+1}})$} \label{ss53}
Introduce the operator $\Op(\V_B * \delta_{\sqrt{2q+1}})$. We have
$$
\widehat{\V_B * \delta_{\sqrt{2q+1}}}(\zeta) = \widehat{\V_B}(\zeta) J_0(\sqrt{2q+1}|\zeta|), \quad \zeta \in \rd.
$$
Note that $J_0$ is an entire function, and $|J_0(r)| \leq 1$, $r \in \re$. On the other hand, $\widehat{\V_B} \in {\mathcal H}_{-2+\rho}^\sharp(\rd)$.
Therefore, $\widehat{\V_B * \delta_{\sqrt{2q+1}}} \in L^{2/(2-\rho)}_w(\rd)$, and by Proposition \ref{gp1} we have $\Op(\V_B * \delta_{\sqrt{2q+1}}) \in S_{2/\rho; w}$. The main result of this subsection is
\begin{pr} \label{p10}
Under the hypotheses of Theorem \ref{th1}  we have
     \bel{53}
     \tr\,\varphi(\lambda_q^{\rho/2} P_q V P_q) =
     \tr\,\varphi(\lambda_q^{\rho/2} \Op(\V_B * \delta_{\sqrt{2q+1}})) + o(\lambda_q), \quad q \to \infty.
     \ee
     \end{pr}
     \noindent
     In the proof of Proposition \ref{p10}, as well as in the next subsection, we will use systematically the following auxiliary result.
     \begin{lemma} \label{gl1}
     Let $\varphi \in C_0^{\infty}(\re\setminus\{0\})$, and let $T = T^*$ and $Q = Q^*$ be compact operators.
     {\rm (i)} Assume $Q \in S_m$, $m \in [2, \infty)$, $T \in S_\ell$, $\ell \in [m, \infty)$.  Then
     \bel{g3}
     \|\varphi(T+Q) - \varphi(T)\|_1 \leq c_{\varphi, m, \ell}\|Q\|_m \left(\|T\|_\ell^{\ell/m'} + \|Q\|_\ell^{\ell/m'}\right)
     \ee
     with $m' = m/(m-1)$ and a constant $c_{\varphi, m, \ell}$ independent of $T$ and $Q$. \\
     {\rm (ii)} Assume $Q \in S_2$, $T  \in S_\ell$, $\ell \in [2, \infty)$.  Then
     \bel{g3a}
     \|\varphi(T+Q) - \varphi(T)\|_1 \leq c_{\varphi, \ell}\|Q\|_2 \left(\|T\|_\ell^{\ell/2} + \|Q\|_2\right)
     \ee
     with a constant $c_{\varphi,\ell}$ independent of $T$ and $Q$.\\
     {\rm (iii)} Assume $Q \in S_2$, $T  \in S_{\ell; w}$, $\ell \in [2, \infty)$.  Then
     \bel{g4a}
     \|\varphi(T+Q) - \varphi(T)\|_1 \leq c_{\varphi, \ell; w}\|Q\|_2 \left(\|T\|_{\ell, w}^{\ell/2} + \|Q\|_2\right)
     \ee
     with a constant $c_{\varphi, \ell; w}$ independent of $T$ and $Q$.
     \end{lemma}
     \begin{proof} By \cite{ps}, for each $p \in (1,\infty)$ and each Lipschitz function $f : \re \to \C$ satisfying
     $$
     \|f\|_{\rm Lip} : = \sup_{\lambda \in \re, \, \mu \in \re, \, \lambda \neq \mu} \frac{|f(\lambda) - f(\mu)|}{|\lambda - \mu|} < \infty,
     $$
     there exists a constant $c_p$ such that
     \bel{g200}
     \|f(M+N) - f(M)\|_p \leq c_p \|f\|_{\rm Lip} \|N\|_p
     \ee
     for each self-adjoint $M$, and each $N = N^* \in S_p$. Further, assume $f \in C_0^\infty(\re\setminus\{0\})$, and set $\delta : = {\rm dist}\,(0, {\rm supp}\,f)$. Then  for any $M = M^* \in S_\infty$ we have
     \bel{g7}
     \|f(M)\|_{s} \leq \max_{\lambda \in \re}|f(\lambda)| \left(\sum_{j: s_j(M) > \delta} 1\right)^{1/s}.
     \ee
     Let $M = M^* \in S_p$, $p \in [1,\infty)$. Then
     \bel{g8}
     \sum_{j: s_j(M) > \delta} 1 \leq \delta^{-p} \|M\|_p^p.
     \ee
      Finally, let $M = M^* \in S_{p, w}$, $p \in (1,\infty)$. Then
     \bel{g9}
     \sum_{j: s_j(M) > \delta} 1 \leq \sum_{j: \|M\|_{p,w} j^{-1/p} > \delta} 1 \leq \delta^{-p} \|M\|_{p,w}^p.
     \ee
     Next, pick a real function $\nu \in C_0^{\infty}(\re\setminus\{0\})$ such that $\nu = 1$ on the support of $\varphi$. Then $\varphi(T) = \varphi(T) \nu(T)$, $\varphi(T+Q) = \varphi(T+Q) \nu(T+Q)$. Assume now $Q \in S_m$, $m \in [2, \infty)$, $T \in S_\ell$, $\ell \in [m, \infty)$.  Then
     \bel{g5}
     \|\varphi(T+Q) - \varphi(T)\|_1 \leq \|\nu(T+Q) - \nu(T)\|_m \|\varphi(T)\|_{m'} + \|\varphi(T+Q) - \varphi(T)\|_m \|\nu(T+Q)\|_{m'},
     \ee
     and \eqref{g3} follows from \eqref{g5}, \eqref{g200} with $M=T$, $N = Q$, and $p=m$, \eqref{g7} with $M=T$ or $M=T+Q$ and $s = m'$, \eqref{g8} with $M = T$ or $M=T+Q$ and $p = \ell$, and the convexity of the function $t \mapsto t^{\ell/m'}$, $t \geq 0$. Further, assume $Q \in S_2$. Then, instead of \eqref{g5}, we can write
     $$
      \|\varphi(T+Q) - \varphi(T)\|_1 \leq
      $$
      \bel{g5a}
     \|\nu(T+Q) - \nu(T)\|_2 \|\varphi(T)\|_{2} + \|\varphi(T+Q) - \varphi(T)\|_2 \left(\|\nu(T)\|_{2} + \|\nu(T+Q) - \nu(T)\|_{2}\right).
     \ee
     If we assume now $T \in S_\ell$ (resp., $T \in S_{\ell, w}$) with $\ell \in [2,\infty)$, then \eqref{g3a} (resp., \eqref{g4a}) follows from \eqref{g5a}, \eqref{g200} with $M=T$, $N=Q$, and $p=2$, \eqref{g7} with $M=T$ and $s=2$, and \eqref{g8} (resp., \eqref{g9}) with $M=T$ and $p = \ell$.
     \end{proof}
     \begin{proof}[Proof of Proposition \ref{p10}] Pick a real radially symmetric $\eta \in C_0^\infty(\rd)$ such that $0 \leq \eta(x) \leq 1$ for all $x \in \rd$, $\eta(x) = 1$ for $|x| \leq 1/2$, $\eta(x) = 0$ for $|x| > 1$. Our first goal is to show that
     \bel{g10}
     \tr\,\varphi(\lambda_q^{\rho/2} P_q V P_q) =
     \tr\,\varphi(\lambda_q^{\rho/2} P_q (1-\eta)\V P_q) + o(\lambda_q), \quad q \to \infty.
     \ee
     Evidently,
     $$
      |\tr\,\varphi(\lambda_q^{\rho/2} P_q V P_q) -
     \tr\,\varphi(\lambda_q^{\rho/2} P_q (1-\eta)\V P_q)| \leq
     $$
     \bel{g11}
     \|\varphi(\lambda_q^{\rho/2} P_q V P_q) - \varphi(\lambda_q^{\rho/2} P_q (1-\eta) V P_q)\|_1 +
     \|\varphi(\lambda_q^{\rho/2} P_q (1-\eta) V P_q) - \varphi(\lambda_q^{\rho/2} P_q (1-\eta) \V P_q)\|_1.
     \ee
     Applying \eqref{g3a} with $\ell > 2/\rho$, $T = \lambda_q^{\rho/2} P_q V P_q$ and $Q = -\lambda_q^{\rho/2} P_q \eta V P_q$, \eqref{31} with $p=2$, and \eqref{32}, we obtain the estimate
     \bel{g15}
     \|\varphi(\lambda_q^{\rho/2} P_q V P_q) -
     \varphi(\lambda_q^{\rho/2} P_q (1-\eta)V P_q)\|_1 = O\left(\lambda_q^{(1+\rho)/2} (\ln{\lambda_q})^{1/2}\right) = o(\lambda_q), \quad q \to \infty.
     \ee
     Similarly, assuming without loss of generality that $\varepsilon \in (0, 1-\rho)$ in \eqref{22}, and then applying \eqref{g3}, with $\ell = m > 2/\rho$, $T = \lambda_q^{\rho/2} P_q (1-\eta) V P_q$,  $Q = -\lambda_q^{\rho/2} P_q (1-\eta) (V - \V) P_q$, as well as \eqref{32}, we obtain
     \bel{g19}
     \|\varphi(\lambda_q^{\rho/2} P_q (1-\eta) V P_q) - \varphi(\lambda_q^{\rho/2} P_q (1-\eta) \V P_q)\|_1 =
     O\left(\lambda_q^{1-\frac{\varepsilon}{2}} (\ln{\lambda_q})\right) = o(\lambda_q), \quad q \to \infty.
     \ee
     Now, \eqref{g11}, \eqref{g15}, and \eqref{g19} imply \eqref{g10}.
     Further, by Proposition \ref{pr50} we have
     \bel{g20}
     \tr\,\varphi(\lambda_q^{\rho/2} P_q (1-\eta)\V P_q) = \tr\,\varphi(\lambda_q^{\rho/2} \Op(((1-\eta)\V)_B * \Psi_q)).
     \ee
     Our next goal is to show that
     \bel{g21}
     \tr\,\varphi(\lambda_q^{\rho/2} \Op(((1-\eta)\V)_B * \Psi_q)) = \tr\,\varphi(\lambda_q^{\rho/2} \Op(\V_B * \delta_{\sqrt{2q+1}})) + o(\lambda_q), \quad q \to \infty.
     \ee
     Similarly to \eqref{g11}, we have
     $$
     |\tr\,\varphi(\lambda_q^{\rho/2} \Op(((1-\eta)\V)_B * \Psi_q)) - \tr\,\varphi(\lambda_q^{\rho/2} \Op(\V_B * \delta_{\sqrt{2q+1}}))| \leq
     $$
     $$
     \|\varphi(\lambda_q^{\rho/2} \Op(((1-\eta)\V)_B * \Psi_q)) - \varphi(\lambda_q^{\rho/2} \Op(((1-\eta)\V)_B * \delta_{\sqrt{2q+1}}))\|_1 +
     $$
     \bel{g22}
     \|\varphi(\lambda_q^{\rho/2} \Op(((1-\eta)\V)_B * \delta_{\sqrt{2q+1}})) - \varphi(\lambda_q^{\rho/2} \Op(\V_B * \delta_{\sqrt{2q+1}}))\|_1.
     \ee
     Applying \eqref{g3a} with $\ell > 2/\rho$,
     $$
     T = \lambda_q^{\rho/2} \Op(((1-\eta)\V)_B * \Psi_q),
     $$
     $$
     Q = \lambda_q^{\rho/2} \left(\Op((1-\eta)\V)_B * \delta_{\sqrt{2q+1}}) -  \Op(((1-\eta)\V)_B * \Psi_q)\right),
     $$
     as well as Proposition \ref{pr50}, \eqref{32}, and Proposition \ref{p5}, we  get
     $$
     \|\varphi(\lambda_q^{\rho/2} \Op(((1-\eta)\V)_B * \Psi_q)) - \varphi(\lambda_q^{\rho/2} \Op(((1-\eta)\V)_B * \delta_{\sqrt{2q+1}}))\|_1 =
     $$
     \bel{g23}
     O\left(\lambda_q^{\frac{1+\rho}{2}-\frac{3}{4}}(\ln{\lambda_q})^{1/2}\right) = o(\lambda_q), \quad q \to \infty.
     \ee
     In order to estimate the second factor at the r.h.s. of \eqref{g22}, we need an estimate of the Hilbert--Schmidt norm of the operator
     $\Op((\eta\V)_B * \delta_{\sqrt{2q+1}}))$. By \eqref{prvb2} and the generalized Young inequality (see e.g. \cite[Section IX.4]{rs2}) we have
     $$
     \|\Op((\eta\V)_B * \delta_{\sqrt{2q+1}}))\|_2^2 = \frac{1}{2\pi} \|\widehat{(\eta\V)_B * \delta_{\sqrt{2q+1}}}\|_{L^2(\rd)}^2 =
     $$
     $$
     = \frac{B^2}{(2\pi)^3} \int_{\rd} |(\hat{\eta} * \hat{\V})(B^{1/2}\zeta)|^2 J_0(\sqrt{2q+1} |\zeta|)^2 d\zeta \leq
     $$
     \bel{g24}
     \frac{B}{(2\pi)^3} \|\hat{\eta} * \hat{\V}\|_{L^2(\rd)}^2 \leq cB \|\hat{\eta}\|_{L^{2/(1+\rho)}(\rd)}^2 \|\hat{\V}\|_{L_w^{2/(2-\rho)}(\rd)}^2
     \ee
     with a constant $c$ which depends only on $\rho$. Applying \eqref{g3a} with $\ell > 2/\rho$ and $T = \lambda_q^{\rho/2}\Op(((1-\eta)\V)_B * \delta_{\sqrt{2q+1}}))$,
     $Q = \lambda_q^{\rho/2}\Op((\eta \V)_B * \delta_{\sqrt{2q+1}}))$, as well as Propositions \ref{pr50} and \ref{p5}, Corollary \ref{f1}, and \eqref{g24}, we get
     $$
     \|\varphi(\lambda_q^{\rho/2} \Op(((1-\eta)\V)_B * \delta_{\sqrt{2q+1}})) - \varphi(\lambda_q^{\rho/2} \Op(\V_B * \delta_{\sqrt{2q+1}}))\|_1 =
     $$
     \bel{g25}
     = O\left(\lambda_q^{(1+\rho)/2}(\ln{\lambda_q})^{1/2}\right) = o(\lambda_q), \quad q \to \infty.
     \ee
     Now, \eqref{g22}, \eqref{g23}, and \eqref{g25} imply \eqref{g21}, while \eqref{g10}, \eqref{g20}, and \eqref{g21} imply \eqref{53}.
       \end{proof}
\subsection{Semiclassical analysis of $\tr\,\varphi(\lambda_q^{\rho/2}\Op(\V_B * \delta_{\sqrt{2q+1}}))$} \label{ss54}
\begin{pr} \label{p11a}
Under the hypotheses of Theorem \ref{th1} we have
    \bel{g26}
    \lim_{q \to \infty} \lambda_q^{-1} \tr \, \varphi(\lambda_q^{\rho/2} \Op(\V_B * \delta_{\sqrt{2q+1}})) = \frac{1}{2\pi B} \int_{\rd} \varphi(B^\rho \vo(x)) dx.
    \ee
    \end{pr}
    \begin{proof}
    Let $s: \re^{2d} \to \C$ be an appropriate Weyl symbol. For $\hbar > 0$ set
    $$
    s_\hbar(x,\xi) : = s(x, \hbar\xi), \quad (x,\xi) \in \re^{2d},
    $$
    and define the $\hbar$-$\Psi$DO
    $\Oph(s) : = \Op(s_\hbar)$.
    Set
    $$
    \tilde{s}_\hbar(x,\xi) : = s(\sqrt{\hbar}x, \sqrt{\hbar}\xi), \quad (x,\xi) \in \re^{2d}.
    $$
    A simple rescaling argument shows that the operators $\Oph(s)$ and $\Op(\tilde{s}_\hbar)$ are unitarily equivalent (see e.g. \cite[Section A2.1]{shu}).
    Set
    $$
    {\bf s}(z) : = B^\rho \, \vo_1(z), \quad z \in \rd.
    $$
    Due to the homogeneity of $\V$ we have
    $$
    \lambda_q^{\rho/2} \, \V_B * \delta_{\sqrt{2q+1}}(z) = {\bf s}((2q+1)^{-1/2}z), \quad z \in \rd.
    $$
    Therefore, the operator $\varphi(\lambda_q^{\rho/2} \Op(\V_B * \delta_{\sqrt{2q+1}}))$ is unitarily equivalent to $\varphi(\Oph({\bf s}))$ with $\hbar : = (2q+1)^{-1}$. Now, in order to prove \eqref{g26}, it suffices to show that
    \bel{g27}
    \lim_{\hbar \to 0} \hbar \tr\,\varphi(\Oph({\bf s})) = \frac{1}{2\pi} \int_{\rd} \varphi({\bf s}(x)) dx
    \ee
    since  $ \int_{\rd} \varphi({\bf s}(x)) dx =  \int_{\rd} \varphi(B^\rho \vo(x)) dx$. If the symbol ${\bf s}$ were regular, then \eqref{g27} would follow from standard semiclassical results (see e.g. \cite[Theorem 9.6]{dsj}). Due to the singularity of ${\bf s}$ at ${\mathbb S}^1$, we need some additional final estimates. Pick the function  $\eta$ defined at the beginning of the proof of Proposition \ref{p10}, and for $r>0$ set $\eta_r(x) : = \eta(r^{-1}x)$, $x \in \rd$. Define the symbols
    $$
    {\bf s}_{1,r}(x) : = B^\rho \, ((1-\eta_r)\V_1)(x), {\bf s}_{2,r}(x) : = B^\rho \, (\eta_r\V_1)(x), \quad x \in \rd,
    $$
so that ${\bf s} = {\bf s}_{1,r} + {\bf s}_{2,r}$. Evidently, ${\bf s}_{1,r} \in {\mathcal S}_1^{-\rho}(\rd)$. By \cite[Theorem 9.6]{dsj},
\bel{g28}
    \lim_{\hbar \to 0} \hbar \tr\,(\varphi(\Oph({\bf s}_{1,r})) = \frac{1}{2\pi} \int_{\rd} \varphi({\bf s}_{1,r}(x)) dx.
    \ee
    On the other hand, estimate \eqref{g4a} with $\ell = 2/\rho$, $T = \Oph({\bf s})$ and $Q = -\Oph({\bf s}_{2,r})$ implies
    $$
    |\tr\,\varphi(\Oph({\bf s})) - \tr\,\varphi(\Oph({\bf s}_{1,r}))| \leq
    $$
    \bel{g29}
    c_{\varphi, \ell; w}\|\Oph({\bf s}_{2,r})\|_2 \left( \|\Oph({\bf s})\|_{2/\rho;w}^{1/\rho} + \|\Oph({\bf s}_{2,r})\|_2\right).
    \ee
    By Proposition \ref{gp1},
    \bel{g30}
    \|\Oph({\bf s})\|_{2/\rho;w} \leq \hbar^{-\rho/2}(2\pi)^{-(1-\frac{2}{\rho})} \|\hat{{\bf s}}\|_{L^{2/(2-\rho)}_w(\rd)}
    \ee
    and, similarly to \eqref{g24},
    \bel{g31}
    \|\Oph({\bf s}_{2,r})\|_2 = \hbar^{-1/2}(2\pi)^{-1/2} \|\widehat{{\bf s}_{2,r}}\|_{L^2(\rd)}
    \leq \hbar^{-1/2} c \|\hat{\V}\|_{L^{2/(2-\rho)}_w(\rd)} \|\widehat{\eta_r}\|_{L^{2/(1+\rho)}(\rd)}.
    \ee
    Finally,
    \bel{g32}
    \|\widehat{\eta_r}\|_{L^{2/(1+\rho)}(\rd)} = r^{1-\rho} \|\widehat{\eta}\|_{L^{2/(1+\rho)}(\rd)}.
    \ee
    As a result, we find that \eqref{g29} -- \eqref{g32} imply the existence of a constant $C$ such that
    the estimate
    \bel{g33}
    |\tr\,\varphi(\Oph({\bf s})) - \tr\,\varphi(\Oph({\bf s}_{1,r}))| \leq C \hbar^{-1} r^{1-\rho}
    \ee
    is valid for each $\hbar > 0$ and $r \in (0,1)$. Now, \eqref{g28} and \eqref{g33} yield
    $$
    \frac{1}{2\pi} \int_{\rd} \varphi({\bf s}_{1,r}(x)) dx - Cr^{1-\rho} \leq
    $$
    $$
     \liminf_{\hbar \downarrow 0} \hbar \tr\,\varphi(\Oph({\bf s})) \leq \limsup_{\hbar \downarrow 0} \hbar \tr\,\varphi(\Oph({\bf s})) \leq
     $$
     $$
     \frac{1}{2\pi} \int_{\rd} \varphi({\bf s}_{1,r}(x)) dx + Cr^{1-\rho}.
     $$
     Letting $r \downarrow 0$, and taking into account that
     $$
     \lim_{r \downarrow 0}\int_{\rd} \varphi({\bf s}_{1,r}(x)) dx = \int_{\rd} \varphi({\bf s}(x)) dx,
     $$
     we obtain \eqref{g27}, and hence \eqref{g26}.
      \end{proof}
      \noindent
Putting together \eqref{51}, \eqref{52}, \eqref{53}, and \eqref{g26}, we arrive at \eqref{24} which completes the proof of Theorem \ref{th1}. \\

\noindent
       {\bf Acknowledgements.}
 This article was initiated in the Fall of 2012 when the authors participated in the Programme ``{\em Hamiltonians in Magnetic Fields}"
 at {\em Institut Mittag-Leffler}, Djursholm, Sweden. G. Raikov
 acknowledges the financial support of IML.
 The authors were partially supported
 by {\em N\'ucleo Cient\'ifico ICM} P07-027-F ``{\em
Mathematical Theory of Quantum and Classical Magnetic Systems"}, as well as
 by the Chilean Science Foundation  {\em Fondecyt} under Grant 1130591.\\

\noindent
{\sc Tom\'as Lungenstrass}\\
 Departamento de Matem\'aticas, Facultad de
Matem\'aticas,\\ Pontificia Universidad Cat\'olica de Chile,
Vicu\~na Mackenna 4860, Santiago de Chile\\
E-mail: tlungens@mat.puc.cl\\

\noindent
{\sc Georgi Raikov}\\
 Departamento de Matem\'aticas, Facultad de
Matem\'aticas,\\ Pontificia Universidad Cat\'olica de Chile,
Vicu\~na Mackenna 4860, Santiago de Chile\\
E-mail:  graikov@mat.puc.cl\\
\end{document}